\documentclass[12pt]{article} 

\usepackage{graphicx}
\usepackage{amsmath,amssymb,textcomp,graphicx,amsthm} 
\usepackage{epsfig}
\usepackage{color}
\usepackage[english]{babel} 
\usepackage[alphabetic]{amsrefs}
 \usepackage{verbatim}
\usepackage[latin1]{inputenc}  
\usepackage[allcolors = blue, colorlinks = true]{hyperref}
\usepackage{enumerate}

\newcommand{\kom}[1]{}
\renewcommand{\kom}[1]{{\bf [#1]}}

\numberwithin{equation}{section}

\newcommand{\e}{\varepsilon}
\newcommand{\<}{\langle}
\renewcommand{\>}{\rangle}

\newcommand{\R}{\mathbb{R}}

\newtheorem{thm}{Theorem}
\newtheorem{de}[thm]{Definition}
\newtheorem{lem}[thm]{Lemma}

\newtheorem{cor}[thm]{Corollary}

\newcommand{\dd}{\partial}
\renewcommand{\d}{\hspace{0.05em}\textup{d}}

\renewcommand{\div}{\operatorname{div}}

\newcommand{\ont}{\int\limits_0^T\!\!\!\int\limits_\Omega}

\newcommand{\abs}[1]{\left|#1\right|}
\newcommand{\Om}{\Omega}

\title{Trudinger's Parabolic Equation}

\author{Peter Lindqvist, Mikko Parviainen, Saara Sarsa}
\date{\phantom{abc}}
\begin{document}

\maketitle

\begin{abstract}
\noindent \textsf{We study the uniqueness of non-negative solutions of the equation
$$
\dd_t\left(|u|^{p-2}u\right)\,=\, \div(|\nabla u|^{p-2}\nabla u).
$$
Basic estimates are derived with the Galerkin method.}
\end{abstract}

\bigskip 

{\small \textsf{AMS Classification 2020}: 35A02, 35K65, 35K55} 



\section{Introduction}
 The  doubly non-linear evolutionary equation
\begin{equation}
\label{eq:maineq}
\frac{\dd\,(|u|^{p-2}u)}{\dd t}\,=\, \div (|\nabla u|^{p-2}\nabla u)   \quad \text{in }\quad \Omega\times (0,T),
\end{equation}
where $\Omega$ is a domain in $\R^n$, was introduced by N. Trudinger in \cite{T}. He observed that, due to the homogeneous structure, it enjoys the property that no ``intrinsic scaling'' is needed in its Harnack Inequality. However, especially when it comes to sign-changing solutions, little is known in contrast to the  current situation for the evolutionary $p$-Laplace equation
$$  \frac{\dd u}{\dd t}\,=\, \div (|\nabla u|^{p-2}\nabla u).$$
\begin{itemize}
\item The natural uniqueness seems to be unsettled for the ordinary Cauchy-Dirichlet boundary value problem in space$\times$time  cylinders.
\item The local boundedness of the gradient
  $$\nabla u\,=\, \Big(\frac{\dd u}{\dd x_1},\frac{\dd u}{\dd x_2},\dots,\frac{\dd u}{\dd x_n}\Big)$$ is problematic\footnote{The case $\inf\{u\} > 0$ is established in the recent \cite{LPS}.}, not to mention further regularity questions.
\end{itemize}

Our object is the uniqueness. Some surprising complications arise, because regularized equations like
$$\frac{\dd}{\dd t}\Bigl((u^2+\e^2)^\frac{p-2}{2}u\Bigr)\,=\,\nabla\! \cdot\!\bigl((|\nabla u|^2+\e^2)^\frac{p-2}{2}\nabla u \bigr)$$
do not seem to satisfy the assumptions required in the classical theory presented in the monograph~\cite{LUS}, written by O. Ladyzhenskaya, N. Uraltseva, and V. Solonnikov. A favorable exception is the one-dimensional case $n=1$, see \cite[Theorem 5.2, Chapter VI.5, p.564]{LUS}. Considerations, referring to ``classical theory'' often suffer from this lack.

 For the case $u\,\geq\,0$, we have a slightly stronger uniqueness theorem than  previously known results, which are for $\inf u>0$.
\begin{thm}\label{mainresult} \sloppy Let $p \geq 2.$ A non-negative weak solution  $u \in C(\Omega_T)\,\cap\,L^p(0,T;W^{1,p}(\Omega))$   with Cauchy-Dirichlet  boundary values  $\psi \in C^2(\overline{\Omega_T})$   is unique. Moreover, the Sobolev derivative $\dd_t(u^{p-1})$ exists and belongs to $L^2(\Omega_T).$
\end{thm}
The result implies that less regular solutions with these boundary values $\psi$ cannot exist. It is curious that sign-changing solutions are employed in our proof in Section  \ref{gammaproof}. Yet, the case $u\,\geq\,0$ is not exhausted. What if $\psi$ merely belongs to $ C^1(\overline{\Omega_T})$ or only to $C(\overline{\Omega_T})$? Our approach is based on Galerkin's method in \cite{AL}. Unfortunately, the  estimates that we need are omitted in \cite{AL}. No doubt, it is known to the experts, but in Section \ref{approximative} we have worked out fundamental  estimates the whole way from the Galerkin approximations. This standard procedure has its advantages: it yields existence and explicit  estimates also for sign-changing solutions. Another approach is by Rothe's method as done  for Trudinger's equation in \cite{MN}, there restricted to zero lateral boundary values.  In \cite{S}, systems are treated.

Some of our exposition is valid for $1<p<\infty$, but strictly speaking Section~\ref{approximative} is written only for the case $p>2$. The case $p=2$ is simpler.

\section {Definitions and known properties}

We use standard notation. Let $\Omega_T\,=\,\Omega \times (0,T)$, where $\Omega \subset \mathbb{R}^n$ is an open, bounded domain with Lipschitz-regular \emph{parabolic boundary}
$$\dd_{p}\Omega_T\,=\, (\overline{\Omega} \times \{0\})\,\cup\,(\dd \Omega \times [0,T]).$$

\begin{de}\label{solution} We say that $u \in L^p(0,T;W^{1,p}(\Omega))$ is a weak solution   in $\Omega_T$ if
  $$\ont\left(-|u|^{p-2}u\,\phi_t\,+\,\<|\nabla u|^{p-2}\nabla u,\nabla \phi\> \right)\d x\d t\,\,=\,\,0$$
  for all $\phi \in C_0^{\infty}(\Omega_T).$
\end{de}
By advanced regularity theory, the weak solutions are locally H\"{o}lder continuous, cf.\ \cite{BDL}. We shall always use the continuous representative: $u \in C(\Omega_T).$
If the test function $\zeta$ vanishes only on the lateral boundary,
the equation reads
\begin{equation}\label{zetaeq}
 \Bigl{ \vert}_{t_1}^{t_2}  \int\limits_{\Omega} \zeta |u|^{p-2}u \d x \,=\,
    \int\limits_{t_1}^{t_2}\!\!\!\int\limits_\Omega\left(-|u|^{p-2}u\,\zeta_t+\<|\nabla u|^{p-2}\nabla u,\nabla \zeta\> \right)\d x\d t.
    \end{equation}

The \emph{maximum principle} is valid\footnote{Consider a function $u$ and its boundary maximum $M$, which both are solutions. Then, test both regularized weak formulations by $\eta_{\epsilon} (t)\big((|u|^{p-2} u)_h-(|M|^{p-2} M)_h\big)_+$, where $\eta_\epsilon$ is a time cut-off, and the result follows.}:
$$\min_{\dd_p \Omega_T} \{u\}\,\leq\,u\,\leq \,\max_{\dd_p \Omega_T}\{u\}\quad \text{if}\quad u \in C(\overline{\Omega_T}).$$

Harnack's inequality was proved in \cite[Theorem 2.6]{KK} for strictly positive solutions $u>0$, see also \cite{GV} and \cite{IMM}. The version given in \cite[Theorem~B.1]{BDL} allows non-negative solutions $u \geq 0$ and implies the following:

\smallskip

If $u\geq0$ in $\Omega_T$ and if $u(x_0,t_0) = 0$ at some interior point $(x_0,t_0) \in \Omega_T$, then
$$u(x,t)\,=\,0\quad \text{in} \quad \Omega \times (0,t_0).$$

\noindent\textbf{Boundary values.} We shall consider the boundary value problem $u=\psi$ on the parabolic boundary $\dd _p \Omega_T$, where the prescribed boundary values are induced by a given function
$$\psi \in C(\overline{\Omega_T}) \cap L^p(0,T;W^{1,p}(\Omega)).$$
A weak solution $u$ in $L^p(0,T;W^{1,p}(\Omega))$ is supposed to take the boundary values (at least) in the following sense:
\begin{description}
\item[\qquad(i)]\, $u-\psi \in L^p(0,T;W^{1,p}_0(\Omega))$
\item[\qquad(ii)]\, $\underset{t \to 0}{ \lim}\, \int\limits_{\Omega}\abs{|u(x,t)|^{p-2}u(x,t) -|\psi(x,t)|^{p-2}\psi(x,t)}\d x \,=\,0$.
\end{description}

A weak solution with boundary values $\psi$ exists by \cite{AL}, see also \cite{S}, and with zero lateral boundary values \cite{MN}. If $u\in C(\overline{\Omega_T})$, then conditions (i) and (ii) can be replaced by simply saying that $u = \psi$ on $\dd _p \Omega_T$. If $\psi$ is H\"{o}lder continuous in $\overline{\Omega} \times [0,T)$, so is $u$,  and $u = \psi$ on  $\dd _p \Omega_T$, cf.\ \cite{BDL}.

  The central problem is the following uniqueness question.
  \medskip

\noindent  \textbf{Problem.} \emph{Suppose that the weak solutions $u_1, u_2 \in  L^p(0,T;W^{1,p}(\Omega))$ have the same boundary values in the sense that}
 \begin{align*}
\text{{(i)}}&\,\, u_2-u_1 \in L^p(0,T;W^{1,p}_0(\Omega)),\\
\text{{(ii)}}&\,\, \lim_{t \to 0}  \int\limits_{\Omega}\abs{|u_2(x,t)|^{p-2}u_2(x,t) -|u_1(x,t)|^{p-2}u_1(x,t)}\d x \,=\,0.
\end{align*}

\emph{ Is it then true that $u_2=u_1$ in $\Omega_T$?}
 \medskip

The assumption $u_1, u_2 \in C(\overline{\Omega_T})$ does not seem to simplify  this difficult problem. Some special cases have been solved:
\begin{itemize}
\item  If $\inf\{u_1\} > 0$ and $\inf\{u_2\} > 0$, then $u_1 =u_2$. See \cite{IMJ} or \cite{LL}.
  \item If $\inf\{u_1\} > 0$ and $u_2\geq 0$, then $u_1 =u_2$. See \cite{LL}.
\item  If $$\dfrac{\dd\,}{\dd t}\bigl(|u_2(x,t)|^{p-2}u_2(x,t) -|u_1(x,t)|^{p-2}u_1(x,t)\bigr) \in L^1(\Omega_T),$$ then $ u_1 =u_2.$
\item $u_1 = 0\quad \Rightarrow\quad u_2 = 0$.
  \item If $u_1, u_2 \in L^p(0,T;W^{1,p}_0(\Omega))$, then $u_1 =u_2$. See \cite{O}.
\end{itemize}

 The next-to-last case comes by taking Stekloff averages. It is a special case of the much deeper last case, which is valid for sign-changing solutions.

It is instructive to prove the third case, where time derivatives are available. To this end, let $H_{\delta}(s)$ denote the usual approximation
\begin{equation*}H_{\delta}(s) \,=\, \begin{cases} 0,\,s\leq 0 \\
    \dfrac{s}{\delta},\,0 <s<\delta\\
    1,\, s\geq \delta
\end{cases} \end{equation*}
of the Heaviside function, see for instance \cite{IMJ}.  We can use the test  function $\zeta = H_{\delta}\bigl(u_2-u_1\bigr)$ in  both equations:
$$\int\limits_{t_1}^{t_2}\!\!\int\limits_{\Omega}\Bigl(\zeta\, \dd_t(|u_j|^{p-2}u_j)\,+\, \big\langle |\nabla u_j|^{p-2}\nabla u_j,\nabla \zeta \big\rangle \bigr) \d x \d t\,=\,0, \quad j = 1,2.$$

 Subtracting them, we see that
\begin{gather*}
  \int\limits_{t_1}^{t_2}\!\!\int\limits_{\Omega}\dd_t\left(|u_2(x,t)|^{p-2}u_2(x,t) -|u_1(x,t)|^{p-2}u_1(x,t)\right)H_{\delta}\bigl(u_2-u_1\bigr) \d x \d t\\
  = - \frac{1}{\delta}  \int\!\!\int \big\langle |\nabla u_2|^{p-2}\nabla u_2 -   |\nabla u_1|^{p-2}\nabla u_1,\nabla u_2 - \nabla u_1 \big\rangle \d x \d t\,\,\leq\,\, 0,
\end{gather*}
where the integration on the right is taken over the positivity set of $u_2-u_1$, and we used the vector inequality $\big\langle |b|^{p-2}b-|a|^{p-2}a,b -a \big\rangle \,\geq\,0$. Letting $\delta \to 0$, we obtain
$$ \int\limits_{t_1}^{t_2}\!\!\int\limits_{\Omega}\dd_t\left(|u_2(x,t)|^{p-2}u_2(x,t) -|u_1(x,t)|^{p-2}u_1(x,t)\right)^+ \d x \d t\,\, \leq \,\, 0.$$

Upon integration,
\begin{align*}
  &\int\limits_{\Omega}\left(|u_2(x,t_2)|^{p-2}u_2(x,t_2) -|u_1(x,t_2)|^{p-2}u_1(x,t_2)\right)^+ \d x \\  \leq \,&\int\limits_{\Omega}\left(|u_2(x,t_1)|^{p-2}u_2(x,t_1) -|u_1(x,t_1)|^{p-2}u_1(x,t_1)\right)^+ \d x \,\rightarrow \,0,
\end{align*}when $ t_1\, \to\,0$.
We conclude that $u_2\leq u_1$ in $\Omega_T$. By symmetry,  $u_1\leq u_2$.

\medskip
Actually, by this we have also proved the analogous \emph{comparison principle}: If $u_2\leq u_1$ holds on the parabolic boundary, so too does it in the whole domain.
\begin{cor}[Comparison Principle]\label{comparison} Suppose that $u_1$ and $u_2$ are weak solutions   and belong to $L^p(0,T;W^{1,p}(\Omega))$, and that
  $$\frac{\dd}{\dd t}\big(|u_2(x,t)|^{p-2}u_2(x,t)-|u_1(x,t)|^{p-2}u_1(x,t)\big) \,\in\, L^1(\Omega_T).$$
  
  If $(u_2-u_1)^+ \in L^p(0,T;W^{1,p}_0(\Omega))$ and
  \begin{equation*}
  \lim_{t\to 0}\,\int\limits_{\Omega}(|u_2(x,t)|^{p-2}u_2(x,t)-|u_1(x,t)|^{p-2}u_1(x,t)\big)^+\d x\,=\,0,
  \end{equation*} then $u_2\leq u_1$.
  \end{cor}

\section{Convergence and a Galerkin estimate}

From now on we assume that the function $\psi$ representing the Cauchy-Dirichlet boundary data belongs to $C^2(\overline{\Omega_T})$. There exists  one weak solution $u \in C(\overline{\Omega_T})$ for which the Sobolev derivative \break$\dd_t(|u|^{p-2}u)$ exists  and \begin{equation}
\label{timereg}
\ont\Bigl(|\dd_t(|u|^{p-2}u)|^2 + |\nabla u|^p\Bigr)\d x \d t\, < \, \infty
\end{equation}
and $u = \psi$ on the parabolic boundary $\dd_{p}\Omega_T$. Unfortunately, for sign-changing solutions, we cannot exclude the possibility that less regular solutions exist. The existence comes by Galerkin's method, see \cite{AL}. In Section \ref{approximative},  we shall work out expedient explicit estimates, omitted in \cite{AL}.
Solutions obeying the bound (\ref{timereg}) form a closed class. For lack of a better name, we call them \emph{t-regular}. Within this class the comparison principle is valid.

\begin{lem}\label{convergence} Suppose that the t-regular solutions $u_1,u_2,u_3,\ldots$ have the uniform bounds $\|u_k\|_{L^{\infty}(\Omega_T)}\,\leq \,C$ and
    \begin{equation}\label{M}
        \ont\Bigl(|\dd_t(|u_k|^{p-2}u_k)|^2 + |\nabla u_k|^p\Bigr)\d x \d t\,\, \leq \,\, M, \end{equation}
      when $k = 1,2,3,\ldots.$ Then, via a subsequence, $u_k$ converges in $L^p(0,T;W^{1,p}_{\text{loc}}(\Omega))$ to some function $u$ satisfying the same bounds. The Sobolev derivative $\dd_t(|u|^{p-2}u)$   exists  and belongs to $L^2(\Omega_T).$  Moreover, $u$ is a weak solution:
      $$\ont\Bigl(\zeta\,\dd_t(|u|^{p-2}u) +\big\langle |\nabla u|^{p-2}\nabla u,\nabla \zeta \big\rangle \Bigr)\d x \d t \,\, = \,\, 0,$$
      whenever $\zeta \in L^p(0,T;W_0^{1,p}(\Omega))$.
      \end{lem}
\begin{proof} We obtain a function $u$ such that
  \begin{equation*}
  \begin{cases}
  u_k\, \rightharpoonup \, u \quad \text{weakly in}\, L^p(\Omega_T)\\
    \nabla u_k\, \rightharpoonup \, \nabla u \quad \text{weakly in}\, L^p(\Omega_T)\\
    \dd_t(|u_k|^{p-2}u_k)\, \rightharpoonup \, \dd_t(|u|^{p-2}u)  \quad \text{weakly in}\, L^2(\Omega_T)\\
    u_k \,\rightarrow \,  u \quad \text{strongly in}  \, L^p(\Omega_T).
  \end{cases}\end{equation*}
  
  Lower semicontinuity of convex integrals under weak convergence implies that the bound (\ref{M}) also holds for the function $u$. The strong convergence stated above requires a brief argument. First, observing uniform bounds,
  one can extract weakly convergent subsequences for functions and gradients. Moreover, observing the uniform bound for $\dd_t(|u_k|^{p-2}u_k)$, we may extract a strongly converging subsequence for $|u_k|^{p-2}u_k$ to some limit $v$ by the Lions-Aubin theorem. Since the convergence is strong, we obtain for some $w$ that
    \begin{align*}
       0\le &\ont (\abs{w}^{p-2}w-\abs{u_k}^{p-2}u_k)(w-u_k)\d x\d t\\
       \to& \ont (\abs{w}^{p-2}w-v)(w-u)\d x\d t.
    \end{align*}
    
    Then, redefining $w$ as $u+\delta \phi$ for $\delta>0$ and passing to $\delta\to 0$, we obtain
    \begin{align*}
        0\le \ont (\abs{u}^{p-2}u-v) \phi \d x\d t,
    \end{align*}
    and thus $v=\abs{u}^{p-2}u$ a.e., and further $\abs{u_k}^{p-2}u_k\to \abs{u}^{p-2}u$ strongly in $L^p(\Omega_T)$.
      Then the claimed strong convergence follows by using an algebraic and H\"older's inequalities that give us
\begin{align*}
    \ont \abs{u_k-u}^p \d x \d t\le C \ont \abs{\abs{u_k}^{p-2}u_k-\abs{u}^{p-2}u}^{\frac{p}{p-1}} \d x \d t,
\end{align*}
when $p>2$.

  We claim that  $\nabla u_k \to \nabla u$ locally strongly in $L^p$. We choose a function $\theta=\theta(x)$ in $C^\infty_0(\Om)$, $0\le \theta\le 1$.
   Now, $\theta^p (u_k-u)$
      will do as the test function in the equation for $u_k$. Arranging, we get
  \begin{align*}
    \mathrm{I}_k \,&=\,\ont \theta^{\,p} (u_k-u)\,\dd_t(|u_k|^{p-2}u_k) \d x \d t\\
    &=\,- \ont \theta^p \big\langle|\nabla u_k|^{p-2}\nabla u_k  - |\nabla u|^{p-2}\nabla u, \nabla u_k - \nabla u \big\rangle \d x \d t\\
    &\phantom{=}\,+ \ont \theta^p\big\langle |\nabla u|^{p-2}\nabla u, \nabla u_k - \nabla u \big\rangle \d x \d t\\
    &\phantom{=}\,- \ont (u_k-u)\big\langle |\nabla u_k|^{p-2}\nabla u_k, \nabla \theta^p \big\rangle \d x \d t.
  \end{align*}
  
The last two integrals obviously converge to zero when $k\to \infty$, and so does the integral $\mathrm{I}_k$, because
  $$
  |\mathrm{I}_k|\,\leq\, \sqrt{M}\,\Bigl( \ont \abs{u_k-u}^2 \d x \d t \Bigr)^{\frac{1}{2}}.
  $$
  
  Therefore, we also have that
  \begin{align*}
\lim_{k\to \infty} \ont \theta^p \big\langle|\nabla u_k|^{p-2}\nabla u_k  - |\nabla u|^{p-2}\nabla u, \nabla u_k - \nabla u \big\rangle \d x \d t\,=\,0.
\end{align*}

   An elementary inequality shows that this is possible only if \break$\theta\nabla u_k \to \theta\nabla u$ (strongly) in $L^p(\Omega_T)$, as we claimed.

  Thus, we may proceed to the limit under the integral sign in
\begin{align*}
\ont \dd_t\bigl(|u_k|^{p-2}u_k\bigr)\, \zeta + \big\langle |\nabla u_k|^{p-2}\nabla u_k, \nabla \zeta \big\rangle \Bigr) \d x \d t\,=\,0
\end{align*}
 to arrive at the claim.
  \end{proof}

 \noindent \textbf{Remark.} We shall only need the special case
  \begin{align*}
  u_k=\psi+\gamma_k\quad \text{on } \partial_p \Om_T,
\end{align*}
with \emph{constants} $\gamma_k\to 0$.
Now the test function $\zeta=u_k-u-\gamma_k$ yields even global strong convergence; $u_k\to u$ in $L^p(0,T;W^{1,p}(\Om))$.
\medskip

The next estimate is written here only for the case $p \geq 2$. The proof is postponed to Section \ref{approximative}.
\begin{thm}[Galerkin estimate] \label{Galerkin}
  Given $\psi \in C^2(\overline{\Omega_T})$, there exists a t-regular solution $u$ attaining the  boundary values $\psi$ and satisfying
  \begin{align}
     \int\limits_0^{T^*}\!\!\! \int\limits_{\Omega}&\Bigl(\Big\vert\frac{\dd}{\dd t}\bigl(|u|^{\frac{p-2}{2}}u\bigr)\Big\vert^2 +|\nabla u|^p + |u|^p\Bigr) \d x \d t  +\int\limits_{\Omega}|\nabla u(x,T^*)|^p\d x\nonumber\\
    \leq\,\,&C_T\Biggl\{\int\limits_{\Omega}\bigl(|\psi(x,0)|^p +   |\nabla \psi(x,0)|^p\bigr)\d x \label{Q}\\
     &\quad + \int\limits_0^{T^*}\!\!\! \int\limits_{\Omega}\Bigl(|\psi|^p + |\nabla \psi|^p +\big\vert \frac{\dd \psi}{\dd t}\big\vert^p +\Big\vert \frac{\dd \nabla \psi}{\dd t}\Big\vert^p \Bigr)\d x \d t \Biggr\}\nonumber
  \end{align}
  for a.e. $T^* \leq T$.
\end{thm}
Notice that the majorant with $\psi$ is stable under perturbations with a constant $\gamma$:  When
$$\psi(x,t)+ \gamma \quad\text{replaces}\quad \psi(x,t),$$
the terms with derivatives remain unchanged.

\section{Proof of Theorem \ref{mainresult}}
\label{gammaproof}

Let $\gamma$ be a (small) constant and consider the t-regular solution $u_{\gamma} = u_{\gamma}(x,t)$ with boundary values $\psi(x,t) +\gamma$. It exists and satisfies the Galerkin estimate in Theorem \ref{Galerkin} uniformly with respect to $\gamma$. We use the three solutions
\begin{equation*}\begin{cases}
    u_{+\gamma}\quad\text{with boundary values}\quad\psi + \gamma\\
   \, u_{0\phantom{+}}\quad\text{with boundary values}\quad\psi + 0\\
    u_{-\gamma}\quad\text{with boundary values}\quad\psi - \gamma
\end{cases}\end{equation*}
for $\gamma > 0$. Here, $u_0$ is the $t$-regular solution (corresponding to $\gamma = 0$).  Notice that $u_{-\gamma}$ \emph{can take negative values}. For $p>2$, we have the uniform bound
$$\ont\Bigl(\Big\vert\frac{\dd}{\dd t}\bigl (|u_{\gamma}|^{p-2} u_{\gamma}\bigr)\Big\vert^2  +|\nabla  u_{\gamma}|^p \Bigr)\d x \d t\,\leq\,M$$
for all $|\gamma| \leq 1$.
 This follows  from Theorem \ref{Galerkin} by using
$$\Big\vert \frac{\dd}{\dd t}\big(|u_{\gamma}|^{p-2}u_{\gamma}\big)\Big\vert^2 = \frac{4}{q^2}|u_{\gamma} |^{p-2}\Big\vert\frac{\dd}{\dd t}\big(|u_{\gamma} |^{\frac{p-2}{2}}u_{\gamma})\Big\vert^2
  \leq 4 \|\psi\! + \!\gamma\|^{p-2}_{\infty}\Big\vert\frac{\dd}{\dd t}\big(|u_{\gamma} |^{\frac{p-2}{2}} u_{\gamma}\big)\Big\vert^2 ,$$
where $1/p+1/q=1$,  and at the last step the maximum principle was used.\footnote{The first inequality follows by selecting $v:=\abs{u}^{\frac{p-2}{2}}u$, and observing that
$\abs{v}^\frac{p-2}{p}v=\abs{u}^{\frac{p}{2}\frac{p-2}{p}}\abs{u}^{\frac{p-2}{2}}u =\abs{u}^{p-2}u.$ Thus $\partial_t(\abs{u}^{p-2}u)=\frac{2(p-1)}{p} \abs{v}^{\frac{p-1}{p}}v_t=\frac{2(p-1)}{p} \abs{u}^{\frac{p-2}{2}}v_t.$
}

  The comparison principle (Corollary \ref{comparison}) for t-regular solutions implies
  $$ u_{-\gamma}\,\leq\, u_0 \,\leq\,   u_{+\gamma}\quad\text{in}\quad \Omega_T.$$
  
  It follows from Lemma \ref{convergence} and the remark thereafter that $ u_{+\gamma}$ converges to some t-regular solution as $\gamma \to 0+$, as does $ u_{-\gamma}$. The limit solution must be $u_0$ due to the uniqueness within the class of t-regular solutions. At least in $L^p(\Omega_T)$ we have (for a subsequence)
  \begin{equation}\label{strangle}
    \lim_{\gamma \to 0} u_{-\gamma}\,=\,u_0\,=\,\lim_{\gamma \to 0} u_{+\gamma}.
  \end{equation}

 Now let $u$ denote an \emph{arbitrary} solution with boundary values $\psi$. From \cite{BDL} we have that $u\in C(\overline{\Omega_T})$. \emph{A priori} we do not have access to its eventual time derivative. We claim that, indeed, $ u = u_0$, which ensures uniqueness. To this end, we show that
  $$   u_{-\gamma}\,\leq\, u \,\leq\,   u_{+\gamma}, $$
  from which the desired uniqueness follows immediately by (\ref{strangle}). Since $ u_{+\gamma}\geq \gamma > 0$ by the minimum principle, we conclude that $ u_{+\gamma} \geq u$ by Theorem 1 in \cite{LL}. \emph{Strict} positivity also on the boundary was decisive in \cite{LL}.

  To deduce that $ u_{-\gamma}\,\leq\, u$, we first reduce the proof to the situation when $u > 0$ in the interior. (In any case $u \geq 0$.) To this end, suppose that $u(x_0,t_0) = 0$ at some point $(x_0,t_0)$ in $\Omega_T$. By Harnack's inequality $u(x,t) = 0$ in $\Omega_{t_0}$, take the supremum $t^* = \sup\{t_0\}$ over all such times. If $t^* = T$ we are done: $u \equiv 0.$ If $t^* < T$, then $u = 0$ in $\overline{\Omega_{t^*}}$. Then, $\psi = 0$ on $\dd_p \Omega_{t^*}$, and so $  u_{-\gamma} = - \gamma$ on  $\dd_p \Omega_{t^*}$. It follows that $u_{-\gamma} = - \gamma$ in the whole $\Omega_{t^*}$, and so the comparison $ u_{-\gamma}\,\leq\, u $ is valid in this domain.
  Thus, it only remains to prove that  $ u_{-\gamma}\,\leq\, u$ in the remaining part $\Omega \times (t^*,T)$. For the proof, we may as well assume that $u> 0$ in $\Omega_T$, i.e. write $t^* =0$.

  By uniform continuity $u_{-\gamma} < u - \frac{\gamma}{2}$ in a zone near the parabolic boundary, therefore the open set
  $$A_{\gamma} \,=\,\left\{ u_{-\gamma}\,>\, u\right\}\,\subset\subset\,\Omega_T$$
  is strictly interior (or empty). It cannot touch the parabolic boundary. (This was the purpose of introducing $u_{-\gamma}$.) Now, $\inf_{A_{\gamma}}\{u\}  > 0.$ In this case, with \emph{strict positivity} the test function
  $$\phi(x,t)\,=\,\eta_{\varepsilon}(t)H_{\delta}\bigl([| u_{-\gamma}|^{p-2} u_{-\gamma} - |u|^{p-2}u]_h\bigr)$$
  in \cite{LL} works well ($\eta_{\varepsilon}(t)$ is a cut-off function, $H_{\delta}$ is an approximation of the Heaviside function, and $[...]_h$ is the usual Stekloff average). Indeed, one can follow pages 405--408 in \cite{LL}, but now using the above test function to prove that $ u_{-\gamma}\,\leq\, u $ in $A_{\gamma}$. We conclude that the inequality
  $ u_{-\gamma}\,\leq\,u$ holds everywhere.\qquad \qquad$\Box$

  \section{Galerkin's method}\label{Method}

  Select (convenient) functions $e_j = e_j(x)$ in $W^{1,p}_0(\Omega) \cap C(\overline{\Omega})$ so that the subspace
  $$\mathrm{span}\left\{e_1,e_2,\ldots,e_m,\ldots\right\}$$
  is dense in $W^{1,p}_0(\Omega)$. They are independent. Divide the time interval by the points
  $$0,h,2h,\ldots,mh =t_h,$$
  where $t_h = T$ (or $t_h \to $ some selected $t$).
  Replace the boundary values $\psi \in C^2$ by the averages
  $$ \psi_h(x,t)\,=\,\frac{1}{h} \!\int\limits _{(k-1)h}^{kh}\!\!\psi(x,\tau)\d \tau,\qquad (k-1)h<t\le kh,$$
  in each subinterval $((k-1)h,kh].$  Then, $\psi_h$ \emph{is time independent in each subinterval}. We construct  approximative solutions
\begin{equation*}
  u_{hm}(x,t)\,=\,\psi_h(x,t)+ \sum_{j=1}^m\alpha_{hm,j}(t)e_j(x),\qquad \alpha_{hm,j} \in L^{\infty}(0,T),
\end{equation*}
with unknown \emph{ piecewise time independent  coefficients} $\alpha_{hm,j} $, which will be determined below. Thus, $u_{hm}(x,t)$ too is time independent in each subinterval. We use the notation
$$\dd^{-h}_tf(x,t)\,=\,\frac{f(x,t)-f(x,t-h)}{h}$$
for the  difference ratios backward in time.

\begin{lem} \label{alpha} There exist $\alpha_{hm,j}$ so that the equation \begin{gather*}
  \int\limits_{\Omega} \dd^{-h}_t\bigl(| u_{hm}(x,t)|^{p-2}u_{hm}(x,t)\bigr)\zeta(x) \d x\\
  \hspace{6em}+ \int\limits_{\Omega}|\nabla  u_{hm}(x,t)|^{p-2}\big\langle \nabla  u_{hm}(x,t),\nabla \zeta(x)\big\rangle \d x\,\,=\,\,0
  \end{gather*}
  holds for a.e. fixed $t$ when $\zeta \in \mathrm{span}\{e_1,e_2,\ldots,e_m\}.$
\end{lem}
\begin{proof} By extension, we define $\psi(x,t)=\psi(x,0)$ when $-h\leq t\leq 0$ so that the recursive procedure
  \begin{align*}& \frac{1}{h}
    \int\limits_{\Omega} | u_{hm}(x,t)|^{p-2}u_{hm}(x,t)\zeta(x) \d x\\
   & \qquad  +\int\limits_{\Omega}|\nabla  u_{hm}(x,t)|^{p-2}\big\langle \nabla  u_{hm}(x,t),\nabla \zeta(x)\big\rangle \d x\\
    &= \, \frac{1}{h}
    \int\limits_{\Omega} | u_{hm}(x,t-h)|^{p-2}u_{hm}(x,t-h)\zeta(x) \d x
  \end{align*}
  can start from $u_{hm}(x,t) = \psi(x,0)$ when $t \leq 0.$  Thus, $\alpha_{hm,j}(t) = 0$ when $t \leq 0$. So far, we have followed \cite{AL}.

  We proceed by recursion. Suppose that we have obtained the coefficients \break$\alpha_{hm,j}(x,t-h)$. Then, the existence of $\alpha_{hm,j}(x,t)$ will come from the Euler-Lagrange equation of the variational integral
  \begin{gather*}\mathcal{F}_{hm}(w)\,=\,\frac{1}{p}\int\limits_{\Omega}|\nabla(w(x)+\psi_h(x,t))|^p\d x\\
  \hspace{-.3 em}  +\frac{1}{h}\int\limits_{\Omega}\Bigl(\frac{|w(x)\!+\!\psi_h(x,t)|^p}{p} - | u_{hm}(x,t\!-\!h)|^{p-2}u_{hm}(x,t\!-\!h)(w(x)\!+\!\psi_h(x,t))\Bigr)\d x.
  \end{gather*}
  
The integral is time independent in the subinterval $(k-1)h<t<kh$! Keep $t$ in this interval.  The integral has a minimum among all $w = w(x)$ belonging to the subspace
  $$ V_m\,=\, \mathrm{span}\{e_1,e_2,\ldots,e_m\}$$
  when $h,m$ (and $t$) are kept fixed. Indeed,
  \begin{align*}
    \mathcal{F}_{hm}(w)\,&\geq\,\frac{1}{p}\int\limits_{\Omega}|\nabla(w+\psi_h)|^p\d x\\
    &\hspace{1 em}+\frac{1}{h}\int\limits_{\Omega}\Bigl(\frac{|w+\psi_h|^p}{2p} - c(p)| u_{hm}(x,t-h)|^{p} \Bigr) \d x\\
    &\geq -\frac{c(p)}{h}\int\limits_{\Omega}|u_{hm}(x,t-h)|^{p}\d x
  \end{align*}
  implies  that
  $$\mu\,=\, \inf_{w}\mathcal{F}_{hm}(w)\, > - \infty.$$
  
  Choose a minimizing sequence $w_1,w_2,w_3,\ldots$ so that
  $$\mu\,\leq\, \mathcal{F}_{hm}(w_i)\,< \mu +\frac{1}{i}.$$
  
  The bound
  \begin{align*}
    \frac{1}{p}&\int\limits_{\Omega}|\nabla(w_i+\psi_h)|^p\d x + \frac{1}{2ph}\int\limits_{\Omega}|w_i+\psi_h|^{p}\d x\\
  & \leq  \mathcal{F}_{hm}(w_i) + \frac{c(p)}{h} \int\limits _{\Omega}|u_{hm}(x,t-h)|^{p}\d x\\
    & \leq \mu + 1 + \frac{c(p)}{h} \int\limits_{\Omega}|u_{hm}(x,t-h)|^{p}\d x
    \end{align*}\\
  shows that we can extract a weakly convergent subsequence
  $w_{i_\nu}\, \rightharpoonup w$ in $W^{1,p}_0(\Omega).$ Clearly, the limit $w$ belongs to the subspace $V_m$, and thus
  $$w(x)\,=\,\sum_{j=1}^m\alpha_je_j(x)$$
  \emph{defines} the  coefficients $\alpha_{hm,j}(t) = \alpha_j$ when $t$ belongs to the \emph{next} time interval; i.e., if $\alpha_{hm,j}(t)$ is known when $t\leq (k-1)h$, then we obtain $\alpha_{hm,j}(t)$ extended to $t\leq kh$. Notice that $\alpha_j$ is constant in the subinterval. We shall see that this is in accordance with Lemma \ref{alpha}.

  By convexity, as before
  $$\mu\,\leq\,\mathcal{F}_{hm}(w)\,\leq\,  \liminf_{\nu \to \infty}\mathcal{F}_{hm}(w_{i_\nu})\,=\,\mu$$
  and so the minimum is attained.
  The Euler-Lagrange equation comes from
  $$\lim_{\varepsilon \to 0}\,\frac{\mathcal{F}_{hm}(w+\varepsilon \zeta) - \mathcal{F}_{hm}(w)}{\varepsilon}\,=\,0,$$
  although only for $\zeta$ restricted to the subspace $V_m$. Written out,
  \begin{gather*}
    \int\limits_{\Omega}|\nabla(w+\psi_h)|^{p-2}\big\langle \nabla(w+\psi_h),\nabla \zeta \big\rangle \d x\\
    + \frac{1}{h} \int\limits_{\Omega}\bigl(|(w+\psi_h)|^{p-2}(w+\psi_h)\,\zeta - |u_{hm}(x,t\!-\!h)|^{p-2}u_{hm}(x,t\!-\!h)\,\zeta\bigr) \d x\,=\,0.
  \end{gather*}
  
  Use the obtained $w$ to define
  $$u_{hm}(x,t) \, = \,\psi_h(x,t) + \sum_{j=1}^m\alpha_{hm,j}(t)e_j(x)$$
  in the ``next'' interval. The equation in Lemma \ref{alpha} is  now verified.
\end{proof}

\subsection{Estimates for the approximative solutions }\label{approximative}

Related to Galerkin's method in \cite{AL}, we shall provide some expedient uniform estimates for $u_{hm}$ in Lemma~\ref{alpha}  aiming at (\ref{main}) below. After establishing this, we can pass to a limit to obtain Theorem~\ref{Galerkin}.

We restrict the calculations to the case $p \geq 2$.
The test function $\zeta = \dd_t^{-h}(u_{hm} - \psi_h)$ is admissible in Lemma \ref{alpha}. Arranging the terms and integrating with respect to $t$, we arrive at
\begin{align*}
  &\ont\dd_t^{-h}\bigl(|u_{hm}|^{p-2} u_{hm}\bigr)\dd_t^{-h} u_{hm}\d x \d t \qquad &\mbox{\phantom{II}I}\\
 & \hspace{1em}+\, \ont \Bigl\langle|\nabla u_{hm}|^{p-2}\nabla u_{hm}, \dd_t^{-h}(\nabla u_{hm})\Bigr\rangle\d x \d t \qquad &\mbox{\phantom{I}II}\\
&  =\ \ont \dd_t^{-h}\bigl(|u_{hm}|^{p-2} u_{hm}\bigr)\dd_t^{-h} \psi_h\d x \d t \qquad &\mbox{III}\\
 & \hspace{1 em}+\,  \ont \Bigl\langle|\nabla u_{hm}|^{p-2}\nabla u_{hm},\dd_t^{-h}(\nabla \psi_h)\Bigr\rangle\d x \d t. \qquad &\mbox{IV}
\end{align*}

We shall estimate the terms, beginning with the main term, i.e, term I.
\begin{description}
\item[\textbf{Term I}]
By a standard inequality,
\begin{gather*}
  \bigl\langle |b|^{p-2}b - |a|^{p-2}a,b-a\bigr\rangle
  \geq \,\frac{4}{p^2}\left\vert  |b|^{\frac{p-2}{2}}b - |a|^{\frac{p-2}{2}}a\right \vert ^2,\\
  \dd_t^{-h}\bigl(|u_{hm}|^{p-2} u_{hm}\bigr)\dd_t^{-h} u_{hm}
  \geq \, \frac{4}{p^2}\left\vert \dd_t^{-h}\bigl(|u_{hm}|^{\frac{p-2}{2}} u_{hm}\bigr)\right \vert ^2.
\end{gather*}
\item[\textbf{Term II}]\, By convexity, \begin{gather*}
  p\,\langle |b|^{p-2}b, b-a \rangle\,\geq\,|b|^p-|a|^p,\\
  \Bigl\langle|\nabla u_{hm}|^{p-2}\nabla u_{hm}, \dd_t^{-h}(\nabla u_{hm})\Bigr\rangle\,\geq\,\frac{1}{p} \, \dd_t^{-h}\bigl(|\nabla u_{hm}|^{p}\bigr).
\end{gather*}
  Integrating the right-hand side and cancelling the overlap, we get the minorant
  \begin{align*}
   & \ont   \dd_t^{-h}\bigl(|\nabla u_{hm}|^{p}\bigr)\d x \d t\\
   & =\,\frac{1}{h}\ont |\nabla u_{hm}(x,t)|^{p}\d x \d t- \frac{1}{h}\ont |\nabla u_{hm}(x,t\!-\!h)|^{p}\d x \d t\\
   & =\,\frac{1}{h}\int\limits_{T\!-\!h}^T\!\!\int_{\Omega}|\nabla u_{hm}(x,t)|^{p}\d x \d t- \frac{1}{h}\int\limits_{-h\,\,\,}^0\!\!\!\int\limits_{\Omega} |\nabla u_{hm}(x,t)|^{p}\d x \d t\\
   & =\, \int\limits_{\Omega}|\nabla u_{hm}(x,T)|^{p}\d x -  \int\limits_{\Omega}|\nabla \psi_h(x,0)|^{p},
  \end{align*}
  since  $u_{hm}(x,t) = u_{hm}(x,T)$, when $T-h\leq t\leq T$
  and the initial data are  $u_{hm}(x,t) = \psi(x,0)$ when $-h < t < 0$.
\item[\textbf{Term III}] \, By Young's inequality,
\begin{align}
\label{eq:III}
\dd_t^{-h}\bigl(|u_{hm}|^{p-2} u_{hm}\bigr)&\dd_t^{-h}\psi_h\,
   \\
   & \leq\,\frac{1}{p}\,| \dd_t^{-h}\psi_h|^p + \frac{1}{q}\,\big\vert\dd_t^{-h}\bigl(|u_{hm}|^{p-2} u_{hm}\bigr)\big\vert^q.\nonumber
\end{align}
  The last term must be absorbed in the main term. To this end, use
  \begin{align*} \Bigl\vert |b|^{p-2}b& - |a|^{p-2}a \Bigr\vert^q\\
    \leq&\, \left\{(p\!-\!1)\Bigl(|b|^{\frac{p-2}{2}} + |a|^{\frac{p-2}{2}}\Bigr)\Bigl\vert
    |b|^{\frac{p-2}{2}}b - |a|^{\frac{p-2}{2}}a\Bigr\vert \right\}^q\\
    \leq&(p\!-\!1)^q\left\{\frac{q}{2}\varepsilon\left\vert|b|^{\frac{p-2}{2}}b - |a|^{\frac{p-2}{2}}a\right\vert^2
    +\Bigl(1\!-\frac{q}{2}\Bigr)\Bigl(\frac{1}{\varepsilon}\Bigr)^{\frac{p}{p-2}}\Bigl(
      |b|^{\frac{p-2}{2}} + |a|^{\frac{p-2}{2}}\Bigr)^{\frac{2q}{2-q}} \right\}\\
      \leq&\,C(p)\left\{\varepsilon \left\vert|b|^{\frac{p-2}{2}}b - |a|^{\frac{p-2}{2}}a\right\vert^2 + \Bigl(\frac{1}{\varepsilon}\Bigr)^{\frac{p}{p-2}}\left(|b|^p+|a|^p\right)\right\},
  \end{align*}
  to see that
  \begin{gather*}
    \frac{1}{q}\,|\dd_t^{-h}\bigl(|u_{hm}|^{p-2} u_{hm}\bigr)|^q \,\leq\,\frac{C(p)\varepsilon}{q}\,\Bigl\vert \dd_t^{-h}\bigl(|u_{hm}|^{\frac{p-2}{2}} u_{hm}\bigr)\Bigr\vert^2\\
    + \frac{C(p)}{q}\Bigl(\frac{1}{\varepsilon}\Bigr)^{\frac{p}{p-2}}\Bigl(|u_{hm}(x,t)|^p+|u_{hm}(x,t\!-\!h)|^p\Bigr).
  \end{gather*}
  
  For the absorbation by term I, we take  $\varepsilon = \varepsilon (p)$ so that
  $$\frac{C(p)\varepsilon}{q}\,=\, \frac{2}{p^2},\quad \text{which is}\quad < \frac{4}{p^2}\quad\text{in term I}.$$
\item[\textbf{Term IV}]\, By Young's inequality,
  $$|\nabla u_{hm}|^{p-2}\bigl\langle \nabla u_{hm}, \dd_t^{-h}(\nabla \psi_h)\bigr\rangle \,\leq\,\frac{1}{q}\,| \nabla u_{hm}|^p + \frac{1}{p}\,|\dd_t^{-h}(\nabla \psi_h)|^p.$$
\end{description}

Now all terms are estimated. Arranging them, we finally arrive at
\begin{align}
  \frac{2}{p^2}&\ont\Bigl\vert \dd_t^{-h}\Bigl(|u_{hm}|^{\frac{p-2}{2}}u_{hm}\Bigr)\Bigr\vert^2\d x \d t + \frac{1}{p}\int\limits_{\Omega}|\nabla u_{hm}(x,T)|^p\d x \label{prel}\\
  &\leq\, \frac{1}{q} \ont |\nabla u_{hm}|^p\d x \d t + A(p)\biggl(\ont |u_{hm}|^p\d x \d t + \int\limits_{\Omega}|\psi_h(x,0)|^p\d x\biggr) \nonumber\\
 & +\, \frac{1}{p}\int\limits_{\Omega}|\nabla \psi_h(x,0)|^p\d x +
  \frac{1}{p}\ont\Bigl(|\dd_t^{-h}\psi_h|^p + |\dd_t^{-h}(\nabla \psi_h) |^p\Bigr)\d x \d t.\nonumber
\end{align}

Before proceeding, we use some basic properties similar to those of Stekloff averages to simplify the $\psi_h$-terms:
\begin{align*}
  \int|\dd_t^{-h}\psi_h(x,t)|^p \d t\,\leq\,\int\Bigl\vert\frac{\dd}{\dd t}\psi(x,t)\Bigr\vert^p\d t,\\
  \int|\dd_t^{-h}\nabla \psi_h(x,t)|^p \d t\,\leq\,\int\Bigl\vert\frac{\dd}{\dd t}\nabla\psi(x,t)\Bigr\vert^p\d t.
\end{align*}

The integrals in the upper bound of (\ref{prel}) containing $ u_{hm}$ have to be bounded independently of $h$ and $m$. A uniform bound for the $L^p$-norm of  $ u_{hm}$  is usually derived by inserting the admissible test function $ u_{hm}-\psi_h$ in Lemma \ref{alpha}. We are content with a  bound coming from the Friedrichs-Sobolev inequality
$$\int\limits_{\Omega}|u_{hm}(x,t)-\psi_h(x,t)|^p\d x \,\leq\,S_p
\int\limits_{\Omega}|\nabla\bigl(u_{hm}(x,t)-\psi_h(x,t)\bigr)|^p\d x,$$
which is to be integrated with respect to $t$.
Also, observe that $|| \psi_h||_{L^p(\Om_T)}\le || \psi||_{L^p(\Om_T)}$ and $|| \nabla \psi_h||_{L^p(\Om_T)}\le ||\nabla  \psi||_{L^p(\Om_T)}$.
We obtain the estimate
\begin{align}\label {Gron}
  \frac{2}{p}&\ont\Bigl\vert \dd_t^{-h}\Bigl(|u_{hm}|^{\frac{p-2}{2}}u_{hm}\Bigr)\Bigr\vert^2\d x \d t + \int\limits_{\Omega}|\nabla u_{hm}(x,T)|^p\d x\nonumber \\
  &\leq\, C\biggl( \ont |\nabla u_{hm}|^p\d x \d t +  \int\limits_{\Omega}|\psi(x,0)|^p +|\nabla \psi(x,0)|^p  \d x\biggr)\\
 &\hspace{1 em} +
  C\ont\Bigl(|\psi|^p+|\nabla\psi|^p+|\frac{\dd}{\dd t}\psi|^p + |\frac{\dd}{\dd t}\nabla \psi |^p\Bigr)\d x \d t.\nonumber
\end{align}

Here, $C = C(p)$. We still need to bound the $L^p$-norm of $\nabla u_{hm}$.
Skipping the first term in inequality (\ref{Gron}), our estimate reads
$$\int\limits_{\Omega}|\nabla u_{hm}(x,T)|^p\d x\,\leq\,  C \ont |\nabla u_{hm}|^p\d x \d t + Q,$$
where $C= C(p)$ and $Q$, representing the $\psi$-integrals, is independent of $h,m,u_{hm}$, and of $T$, say, for $T\leq T_0$. We can regard $T$ as arbitrary by replacing $T=\tau$, where $\tau=0,h,2h,\ldots,T$. Then, one can deduce that the desired estimates also hold for arbitrary $\tau\in(0,T)$ since everything is piecewise constant in time. Then, we can apply Gr\"{o}nwall's inequality on
$$\xi(T)\,=\,\int\limits_{\Omega}|\nabla u_{hm}(x,T)|^p\d x, \quad \xi(T)\,\leq\, C\int\limits_0^T\xi(t)\d t + Q.  $$

It follows that
$$\ont |\nabla u_{hm}|^p\d x \d t \,\leq Q\,C_{T_0} \quad \text{when}\quad T\leq T_0,$$
where the constant $C_{T_0}$ depends only on $T_0$ and the above $C = C(p)$.

Finally, we can write the bounds in the form
\begin{align}\label{main}
  &\ont\biggl(\Bigl\vert \dd_t^{-h}\Bigl(|u_{hm}|^{\frac{p-2}{2}}u_{hm}\Bigr)\Bigr\vert^2 + |\nabla u_{hm}|^p+|u_{hm}|^p\biggr)\d x \d t +\int\limits_{\Omega} |\nabla u_{hm}(x,T)|^p\d x \nonumber \\
  &\qquad\quad\leq\, C(p,T_0)\biggl\{ \int\limits_{\Omega} \bigl(|\psi_{hm}(x,0)|^p + |\nabla \psi_{hm}(x,0)|^p\bigr)\d x\biggr.\\  &\qquad\quad\,\phantom{C(p,T_0)}+\biggl.
\ont\Bigl( |\psi|^p  +  |\nabla \psi|^p + \big\vert \frac{\dd}{\dd t}\,\psi\Big\vert^p + \big\vert \frac{\dd}{\dd t}\nabla\psi\Big\vert^p  \Bigr)\d x \d t\biggr\}, \nonumber
\end{align}
valid when $0 < T \leq T_0$.

\subsection{Proof of Theorem~\ref{Galerkin}}
\paragraph{The Convergence $\mathbf{u_{hm} \to u}$}  Now we have enough of compactness
to conclude that\begin{equation*}
\begin{cases} u_{hm} \to u \quad\text{strongly in} \quad L^p(\Omega_T)\\
  \nabla u_{hm} \rightharpoonup \nabla u \quad\text{weakly in}\quad L^p(\Omega_T)
\end{cases}\end{equation*}
for some function $u$ as $h \to 0$ and $m \to \infty$.
Observe that $u-\psi \in L^p(0,T;W^{1,p}_0(\Omega))$. We need the \emph{strong} convergence
$$
\nabla u_{hm} \,\to\, \nabla u \quad\text{strongly in}\quad L^p(\Omega_T)
$$
in order to verify that $u$ is a weak solution of Trudinger's equation. We prove this by
 selecting functions $v_{hm} \in L^p(0,T;V_m)$ that are piecewise time independent in each interval $((j-1)h,jh)$ and
$$
v_{hm} \to u-\psi \quad\text{strongly in}\quad L^p(0,T;W^{1,p}_0(\Omega)).
$$

In fact,  any function in $L^p(0,T;W^{1,p}_0(\Omega))$ can be approximated by such functions. Then,  the test function
$$\zeta_{hm}\,=\, u_{hm}  - (\psi_h + v_{hm})$$ is admissible
in the equation for $u_{hm}$.
Substituting this test function into the equation in Lemma \ref{alpha}, integrating and arranging the terms, we can write
\begin{align*}
 \ont&\bigl\langle|\nabla u_{hm}|^{p-2}\nabla u_{hm}-|\nabla u|^{p-2}\nabla u, \nabla u_{hm} - \nabla u \bigr\rangle\d x \d t \tag*{I}\\
  = \ont&\bigl\langle|\nabla u_{hm}|^{p-2}\nabla u_{hm}, \nabla (\psi_h+v_{hm}-u)\bigr\rangle\d x \d t \tag*{II}\\
  - & \ont\bigl\langle|\nabla u|^{p-2}\nabla u, \nabla u_{hm} - \nabla u\bigr\rangle\d x \d t \tag*{III}\\
   - &\ont\bigl( u_{hm}-(\psi_h+v_{hm})\bigr)\dd^{-h}_t\bigl(| u_{hm}|^{p-2} u_{hm}\bigr)\d x \d t. \tag*{IV}\\
\end{align*}

The three integrals on the right-hand side converge to zero. To wit, $\mathrm{III} \to 0$ by the weak convergence $\nabla u_{hm} \rightharpoonup \nabla u$. Denoting the majorant in (\ref{main}) by $Q$, we find that
\begin{gather*}|\mathrm{II}|\,\leq\, \Bigl(\ont|\nabla u_{hm}|^p\d x \d t\Bigr)^{\frac{1}{q}}\Bigl(\ont|\nabla(\psi_h+v_{hm}-u)|^p\d x \d t\Bigr)^{\frac{1}{p}}\\\leq\,Q^{\frac{1}{q}}\Bigl(\ont|\nabla(\psi_h+v_{hm}-u)|^p\d x \d t\Bigr)^{\frac{1}{p}} \, \,\to\,\, 0,\\
  |\mathrm{IV}|\,\leq\, \bigl(\ont\bigl\vert\dd^{-h}_t(|u_{hm}|^{p-2}u_{hm})\bigr\vert^2\d x \d t\Bigr)^{\frac{1}{2}}\Bigl(\ont(\psi_h+v_{hm}-u)^2\d x \d t\Bigr)^{\frac{1}{2}}\\ \leq\,Q^{\frac{1}{2}}\Bigl(\ont(\psi_h+v_{hm}-u)^2\d x \d t\Bigr)^{\frac{1}{2}}\,\,\to\,\,0.
\end{gather*}

It also follows that  the first integral $\mathrm{I}$ converges to $0$, and the vector inequality
$$ 2^{2-p}|b-a|^p\,\leq\,\bigl\langle|b|^{p-2}b-|a|^{p-2}a,b-a\bigr\rangle$$
      implies the desired strong convergence $\nabla u_{hm} \to \nabla u$ in $L^p(\Omega_T)$.
      \paragraph{The limit function is a weak solution.} Let $\beta_{M,j}\in C_0^1((0,T))$. The functions
      \begin{equation*}
        \phi(x,t)\,=\,\sum_{j=1}^M\beta_{M,j}(t)e_j(x)
      \end{equation*}
      are dense in the space $L^p(0,T;W^{1,p}_0(\Omega))$. Thus, it is sufficient to verify Definition \ref{solution} for such  test functions. When $m\geq M$, we can use
      $\phi(x,t-h)$ as the test function in the equation for $u_{hm}$ in Lemma \ref{alpha}. Upon integration with respect to $t$, we may write
      \begin{gather*}\ont\phi(x,t-h)\dd_t^{-h}\bigl(|u_{hm}(x,t)|^{p-2}u_{hm}(x,t)\bigr)\d x \d t\\
        + \ont\bigl\langle |\nabla u_{hm}(x,t)|^{p-2}\nabla u_{hm}(x,t) , \nabla \phi(x,t\!-\!h)\bigl\rangle \d x \d t\,=\,0. \nonumber
      \end{gather*}

      The limit of the second integral is evident by the strong convergence of the gradients.
 In  the first integral, the product rule
  \begin{gather*}\dd_t^{-h}\bigl(\phi(x,t)|u_{hm}(x,t)|^{p-2}u_{hm}(x,t)\bigr)\\= \phi(x,t\!-\!h)\dd_t^{-h}\bigl(|u_{hm}(x,t)|^{p-2}u_{hm}(x,t)\bigr) + |u_{hm}(x,t)|^{p-2}u_{hm}(x,t)\dd_t^{-h}\phi(x,t)
  \end{gather*}
\enlargethispage{\baselineskip}
  implies\footnote{The integral with the total difference ratio $\dd^{-h}_t\big(\phi\, |u_{hm}|^{p-2}u_{hm}\big)$ is zero when $h$ is small enough, since $\beta_{M,j}(t)$ has compact support, and thus we have a difference of the same integrals.}
     \begin{align*}
\ont&\phi(x,t-h)\dd_t^{-h}(|u_{hm}(x,t)|^{p-2}u_{hm}(x,t))\d x \d t\\
  &      =\ont\dd_t^{-h}\bigl(\phi(x,t)|u_{hm}(x,t)|^{p-2}u_{hm}(x,t)\bigr)\d x \d t\\
     &\hspace{1 em}   -\ont |u_{hm}(x,t)|^{p-2}u_{hm}(x,t)\dd_t^{-h}\phi(x,t)\d x \d t \,
        \\
&        \to \, - \ont |u|^{p-2}u\frac{\dd\phi}{\dd t}\d x \d t.
\end{align*}

   This verifies that $u$ satisfies the equation in Definition \ref{solution}, that is, $u$ is a weak solution and t-regular. This finishes the proof of Theorem~\ref{Galerkin}.\qquad $\Box$.

\vspace{1 ex}
{\small \textsf{Acknowledgements}: MP is supported by the Research Council of Finland, project 360185. 
SS is supported by Eemil Aaltonen Foundation through the research group `Quasiworld network', as well as by the Research Council of Finland, grant 354241.
The authors would like to thank Verena B\"{o}gelein and Tuomo Kuusi for useful information.}

\bigskip
{\small

\noindent \textsf{Peter Lindqvist\\ Department of
   Mathematical Sciences\\ Norwegian University of Science and
  Technology\\ N--7491 Trondheim, Norway}\\
\textsf{e-mail}: peter.lindqvist@ntnu.no

 \bigskip
\noindent \textsf{Mikko Parviainen \\  Department of Mathematics and Statistics\\ University of Jyv\"askyl\"a \\ 
FI-40014 Jyv\"askyl\"a, Finland}\\
  \textsf{e-mail}: mikko.j.parviainen@jyu.fi
  
\bigskip
\noindent \textsf{Saara Sarsa \\  Department of Mathematics and Statistics\\  University of Jyv\"askyl\"a\\
FI-40014 Jyv\"askyl\"a, Finland}\\
  \textsf{e-mail}: saara.m.sarsa@jyu.fi\\
}


\begin{thebibliography}{ABC}
{\small


\bibitem[AL]{AL}
\newblock H. Alt and S. Luckhaus,
\newblock \emph{Quasilinear elliptic-parabolic differential equations},
\newblock \emph{Mathematische Zeitschrift}, \textbf{183} (1983), 311-341.

\bibitem[BDL]{BDL}
\newblock V. B\"{o}gelein, F. Duzaar and N. Liao,
\newblock \emph{On the H\"{o}lder regularity of signed solutions to a doubly nonlinear equation},
\newblock \emph{Journal of Functional Analysis}, \textbf{281} (2021), Paper No.~109193, 58 pp.

\bibitem[Db]{Db}
\newblock E. DiBenedetto,
\newblock \emph{Degenerate Parabolic Equations},
\newblock Springer-Verlag, Berlin, 1993.

\bibitem[GV]{GV}
\newblock U. Gianazza and V. Vespri,
\newblock \emph{A Harnack inequality for solutions of doubly nonlinear parabolic equations},
\newblock \emph{Journal of Applied Functional Analysis}, \textbf{1}, 2006, 271-284.

\bibitem[IMJ]{IMJ}
\newblock A. Ivanov, P. Mkrtychan and W. J\"{a}ger,
\newblock \emph{Existence and uniqueness of a regular solution of the Cauchy--Dirichlet problem for a class of doubly nonlinear parabolic equations},
\newblock \emph{Journal of Mathematical Sciences}, \textbf{84} (1997),  845-855.

\bibitem[IMM]{IMM}
\newblock P.-A. Ivert, N. Marola and M. Masson,
\newblock \emph{Energy estimates for variational minimizers of a parabolic doubly nonlinear equation on metric measure spaces},
\newblock \emph{Annales Academi\ae\ Scientiarum Fennic\ae\ Mathematica}, \textbf{39} (2014), 711-719.

\bibitem[KK]{KK}
\newblock J. Kinnunen and T. Kuusi,
\newblock \emph{Local behavior of solutions to doubly nonlinear parabolic equations},
\newblock \emph{Mathematische Annalen}, \textbf{337} (2007), 705-728.

\bibitem[LUS]{LUS}
\newblock O. Ladyzhenskaya, N. Uraltseva and V. Solonnikov,
\newblock \emph{Linear and Quasi-Linear Equations of Parabolic Type},
\newblock Translations of Mathematical Monographs, 23. American Mathematical Society, Providence, RI, 1968.

\bibitem[LL]{LL}
\newblock E. Lindgren and P. Lindqvist,
\newblock \emph{On a comparison principle for Trudinger's equation},
\newblock \emph{Advances in Calculus of Variations}, \textbf{15} (2022), 401-415.

\bibitem[LPS]{LPS}
\newblock P. Lindqvist, M. Parviainen and J. Siltakoski,
\newblock \emph{Lipschitz continuity and equivalence of positive viscosity and weak solutions to Trudinger's equation},
\newblock \emph{Math. Ann.}, \textbf{394} (2026), Paper No. 59, 42 pp, arXiv:2502.14670.

\bibitem[MN]{MN}
\newblock M. Misawa and K. Nakamura,
\newblock \emph{Existence of sign-changing weak solutions to doubly nonlinear parabolic equations},
\newblock \emph{The Journal of Geometric Analysis}, \textbf{33} (2023), Paper No. 33, 44 pp.

\bibitem[M]{M}
\newblock J. Moser,
\newblock \emph{A Harnack inequality for parabolic differential equations},
\newblock \emph{Communications on Pure and Applied Mathematics}, \textbf{17} (1964), 101-134.

\bibitem[O]{O}
\newblock F. Otto,
\newblock \emph{$L^1$-contraction and uniqueness for quasilinear elliptic-parabolic equations},
\newblock \emph{Journal of Differential Equations}, \textbf{131} (1996), 20-38.

\bibitem[S]{S}
\newblock L. Sch\"{a}tzler,
\newblock \emph{Existence for singular doubly nonlinear systems of porous medium type with time dependent boundary values},
\newblock \emph{Journal of Elliptic and Parabolic Equations}, \textbf{5} (2019), 383-421.

\bibitem[T]{T}
\newblock N. Trudinger,
\newblock \emph{Pointwise estimates and quasilinear parabolic equations},
\newblock \emph{Communications on Pure and Applied Mathematics}, \textbf{21} (1968), 205-226.


}
\end{thebibliography}
\end{document}